\documentclass{amsart}

\usepackage{amsmath, amssymb}
\newtheorem{theorem}{Theorem}[section]
\newtheorem{lemma}[theorem]{Lemma}
\newtheorem{proposition}[theorem]{Proposition}
\newtheorem{corollary}[theorem]{Corollary}

\theoremstyle{definition}

\newtheorem*{proofoftheorem1}{Proof of Theorem \ref{th1}}
\newtheorem*{proofoftheorem2}{Proof of Theorem \ref{th2}}

\newtheorem*{proofofcorollary2}{Proof of Corollary \ref{cor2}}

\theoremstyle{remark}
\newtheorem{remark}[theorem]{Remark}

\numberwithin{equation}{section}

\newcommand{\Aut}{{\mathrm {Aut}}}

\newcommand{\Irr}{{\mathrm {Irr}}}

\newcommand{\la}{\langle}
\newcommand{\ra}{\rangle}
\newcommand{\MI}{{\mathrm {MI}}}
\newcommand{\mi}{{\mathrm {m.i}}}
\newcommand{\f}{{\mathbb F}}
\newcommand{\St}{{\mathrm{St}}}

\begin{document}
\title{On a generalization of $M$-group}

\author[T. Le]{Tung Le${}^\dag$}

\address{T.L.: School of Mathematical Sciences,
North-West University, Mafikeng Campus, Mmabatho 2735, South Africa}
\address{T.L. (Alternative International Address): Vietnam National University, Ho Chi Minh City, Vietnam}
\email{lttung96@yahoo.com}

\author[J. Moori]{Jamshid Moori$^{\ddag}$}
\address{J.M.: School of Mathematical Sciences,
North-West University, Mafikeng Campus, Mmabatho 2735, South Africa}
\email{jamshid.moori@nwu.ac.za}

\author[H.P. Tong-Viet]{Hung P. Tong-Viet${}^\flat$}
\address{H.P.T-V. : School of Mathematical Sciences,
North-West University, Mafikeng Campus, Mmabatho 2735, South Africa} \email{tvphihung@gmail.com}



\thanks{${}^\dag$ Supported by North-West University (Mafikeng) and the NAFOSTED (Vietnam).\\
${}^\ddag$ Supported by North-West University (Mafikeng) and  a  Competitive Grant from NRF.\\
${}^\flat$ Supported by NRF and North-West University (Mafikeng)}
\date{\today}

\begin{abstract}
{In this paper, we show that if for every nonlinear complex irreducible character $\chi$ of a
finite group $G,$ some multiple of $\chi$ is induced from an irreducible character of some proper
subgroup of $G,$ then $G$ is solvable. This is a generalization of Taketa's theorem on the
solvability of $M$-group.}
\end{abstract}

\maketitle


\section{Introduction and Notation}

All groups in this paper are finite and all characters are complex characters. For a group $G,$ let
$\Irr(G)$ denote the set of all irreducible characters of $G.$ An irreducible character $\chi$ of a
group $G$ is {\em monomial} if it is induced from a linear character of a subgroup of $G,$ that is
$\chi=\lambda^G,$ where $\lambda\in\Irr(U)$ with $\lambda(1)=1$ and $U\leq G.$ A group $G$ is
called an {\em $M$-group} if every irreducible character of $G$ is monomial. A well known theorem
of Taketa  says that all $M$-groups are solvable. (see \cite[Theorem~5.12]{Isaacs}.) There have
been many generalizations of Taketa's theorem in the literature. Observe that if $\chi\in\Irr(G)$
is a monomial character induced from the subgroup $U\leq G$ and the linear character $\lambda\in
\Irr(U),$ then $U/Ker(\lambda)$ is cyclic, in particular $U/Ker(\lambda)$ is solvable. With this
observation, Dornhoff showed in \cite{Dornhoff} that a group $G$ is solvable provided that every
irreducible character of $G$ is induced from an irreducible character of a solvable section of $G.$
More generally,  Isaacs  proved in \cite{Isaacs84} that if every irreducible character
$\chi\in\Irr(G)$ is induced from an irreducible character $\lambda$ of a subgroup $H$ such that
$H/Ker(\lambda)\in \mathfrak{F},$ then $G$ is in $\mathfrak{F},$ where $\mathfrak{F}$ is a class of
groups closed under isomorphisms, subgroups and extensions. If we choose $\mathfrak{F}$ to be the
class of solvable groups, then we obtain the result of Dornhoff mentioned above. A group $G$ is
called a  {\em Quasi-Solvable Induced} (QSI) group if every irreducible character $\chi$ of $G$ has
some multiple which is induced from a character $\lambda$ of a subgroup U with $U/Ker\lambda$
solvable. Recently, K\"{o}nig \cite{Konig} showed that every QSI group is solvable. Obviously, this
is a generalization of both Dornhoff's and Taketa's theorems. The main purpose of this paper is to
remove the solvability assumption on the quotient  $U/Ker(\lambda).$

Recall that an irreducible character of a group is {\em imprimitive} if it is induced from an
irreducible character of some proper subgroup and it is {\em primitive} if it is not induced by any
character of any proper subgroups. For convenience reason, we make the following definitions. A
nonlinear character $\chi\in\Irr(G)$ is called a \emph{multiply imprimitive character} (or {\em
$\mi$ character} for short) induced from the pair $(U,\lambda)$ if there exist a proper subgroup
$U$ of $G$ and an irreducible character $\lambda\in\Irr(U)$ such that $\lambda^G=m\chi$ for some
nonnegative integer $m.$ Moreover, a group $G$ is said to be an {\em $\MI$-group} if every
nonlinear irreducible character of $G$ is an $\mi$ character.

Let $N\unlhd G.$ We write $\Irr(G|N)=\Irr(G)-\Irr(G/N).$ If $N$ is a normal subgroup of $G$ and
every nonlinear irreducible character in $\Irr(G|N)$ is an $\mi$ character, then $G$ is called an
\emph{$\MI$-group relative to} $N.$ We now state our main result.

\begin{theorem}\label{th1} Let $N$ be a normal subgroup of a group
$G.$ If $G$ is an $\MI$-group relative to $N,$ then $N$ is solvable.
\end{theorem}

If we take $N=G',$ then the set $\Irr(G|N)$ is exactly the set of all nonlinear irreducible
characters of $G.$ Now assume that $G$ is an $\MI$-group. Then $G$ is an $\MI$-group relative to
$G'$ and thus by applying Theorem \ref{th1}, we deduce that $G'$ is solvable and so $G$ is
solvable. Therefore, we have proved the following corollary.
\begin{corollary}\label{cor1} Every $\MI$-group is solvable.
\end{corollary}
This gives a positive answer to \cite[Problem~162]{Berk98}.
We also obtain an answer to \cite[Problem~123]{Berk98} as follows.

\begin{corollary}\label{cor2} Let $H\leq G$ be a proper subgroup of a group
$G.$ Suppose that for any $\lambda\in\Irr(H)$ with $\lambda\neq 1_H,$ we have $\lambda^G=m\chi$ for
some $\chi\in\Irr(G)$ and some integer $m\geq 1.$ Then the normal closure of $H$ in $G$ is
solvable. In particular, $H$ is solvable.
\end{corollary}

For the proof of Theorem \ref{th1}, in Section \ref{reduction} we present some results needed for
reducing the problem to a question concerning the existence of a special $\mi$ character in
nonabelian simple groups. Using the classification of nonabelian simple groups, we obtain the
answer to this question which is stated as Theorem \ref{th2} below. This theorem will be verified
in Sections \ref{Lietypegroups}, \ref{Alternatinggroups} and \ref{Sporadicgroups}. Finally, the
proofs of Theorem \ref{th1} and Corollary \ref{cor2} will be carried out in the last section.

\begin{theorem}\label{th2}
If $S$ is a nonabelian simple group, then $S$ has a nonlinear irreducible character which is
extendible to $\Aut(S)$ but it is not an $\mi$ character.
\end{theorem}
\noindent {\bf Notation.} If $G$ is a group, then we write $\pi(G)$ to denote the set of all prime
divisors of the order of $G.$ For a normal subgroup $N$ of $G,$ if $\theta\in\Irr(N),$ then the set
of all irreducible constituents of $\theta^G$ is denoted by $\Irr(G|\theta).$ If $n$ is a positive
integer and $p$ is a prime then $n_p$ and $n_{p'}$ are the largest $p$-part and $p'$-part of $n,$
respectively. The greatest common divisor of two integers $a$ and $b$ is denoted by $\gcd(a,b).$ We
follow \cite{atlas} for notation of simple groups. Other notation is standard.


\section{Reduction to simple groups}\label{reduction}
The following lemma is a modification of  Lemma~2.1 in \cite{Konig}.
\begin{lemma}\label{lem1} Let $K$ and $N$ be normal subgroups of a group $G.$ Suppose that $G$ is an \textrm{MI}-group relative
to $N.$ Then the following hold.

\begin{itemize}
\item[(i)]  $G/K$ is an \textrm{MI}-group relative to $NK/K;$

\item[(ii)] $G$ is an \textrm{MI}-group relative to $K$ provided that
$K\leq N.$
\end{itemize}
\end{lemma}

\begin{proof}
Assume that $\hat{\chi}\in\Irr(G/K|NK/K).$ Then $\hat{\chi}$ can be considered
as a character $\chi$ of $G$ with $K\leq Ker(\chi)=Ker(\hat{\chi}).$ As
$NK/K\nsubseteq Ker(\hat{\chi})$ but $K\leq Ker(\hat{\chi}),$ we
deduce that $N\not\leq Ker(\chi)$ so $\chi\in\Irr(G|N)$ with
$K\leq Ker(\chi).$ Since $G$ is an \textrm{MI}-group relative to $N,$
we deduce that $m\chi=\lambda^G,$ where $U\lneq G,$
$\lambda\in\Irr(U)$ and $m\geq 1.$ We have $$K\leq
Ker(\lambda^G)=\bigcap_{g\in G}(Ker(\lambda))^g$$ and hence $K\leq
Ker(\lambda)\unlhd U.$ Thus  $\lambda$ can be considered as a
character $\hat{\lambda}$ of $U/K.$ For $x\in G,$ we have

\begin{eqnarray*}
\hat{\lambda}^{G/K}(xK)&=&\frac{1}{|U/K|}\sum_{\begin{tabular}{cc}
$yK\in G/K$\\$(xK)^{yK}\in U/K$
\end{tabular}
}\hat{\lambda}((xK)^{yK})
 \\
 &=&\frac{1}{|U|}\sum_{\begin{tabular}{cc}
$y\in G$\\ $x^{y}\in U$
\end{tabular}}\lambda(x^y)= \lambda^G(x)=m\chi(x)=m\hat{\chi}(xK).
 \end{eqnarray*}
Therefore $\hat{\chi}\in\Irr(G/K|NK/K)$ is an $\mi$
character induced from $(U/K,\hat{\lambda}).$ This proves $(i).$ If
$K\leq N\leq G,$  then $(ii)$ is obvious since $\Irr(G|K)\subseteq
\Irr(G|N).$
\end{proof}

\begin{lemma}\label{lem2} Let $\chi\in\Irr(G)$ be an $\mi$
character induced from a subgroup $U\lneq G$ and $\lambda\in\Irr(U)$
with $\lambda^G=m\chi.$ Then

\begin{itemize}
\item[(i)] If $\chi(g)\neq 0$ for some $g\in G,$ then $g^G\cap U\neq
\emptyset.$

\item[(ii)] We have $|G:U|\lambda(1)= m\chi(1),$ $\chi(1)\geq
m\lambda(1)$ and $|G:U|\geq m^2.$
\end{itemize}
\end{lemma}

\begin{proof}

As $\lambda^G=m\chi,$ if $g\in G$ with $\chi(g)\neq 0,$ then
$\lambda^G(g)=m\chi(g)\neq 0.$ By the definition of induced characters,
we have that $xgx^{-1}\in U$ for some $x\in G,$ which proves $(i).$
For $(ii),$ by comparing the degrees, we have
$\lambda^G(1)=m\chi(1),$ which implies that
$|G:U|\lambda(1)=m\chi(1).$ By the Frobenius reciprocity, we have
$m=(\lambda^G,\chi)=(\lambda,\chi_U)$ so $\chi_U=m\lambda+\psi$
for some character $\psi$ of $U.$ Hence
$\chi(1)=m\lambda(1)+\psi(1)\geq m\lambda(1),$ which proves the
second statement of $(ii).$ Finally, we have
$|G:U|\lambda(1)=m\chi(1)\geq m^2\lambda(1),$ which deduces that
$|G:U|\geq m^2.$ \end{proof}

Let $\chi\in\Irr(G)$ be an $\mi$ character induced from
$(U,\lambda),$ that is, $m\chi=\lambda^G$ for some $m\geq 1.$ We
will show that $U$ could be chosen to be a maximal subgroup of $G.$
By definition, $U$ is a proper subgroup of $G,$ and thus there is a
maximal subgroup  $H$ of $G$ that contains $U.$ Let $\mu\in\Irr(H)$
be an irreducible constituent of $\lambda^H.$  Write
$\lambda^H=\mu+\psi,$ where $\psi$ is a character of $H.$ By the
transitivity of character induction, we have that
$(\lambda^H)^G=\lambda^G=m\chi$ so $\mu^G+\psi^G=m\chi.$ Thus
$\mu^G=e\chi$ for some $e\geq 1,$ which means that $\chi$ is an $\mi$
character induced from $(H,\mu),$ where $H$ is maximal in $G$ and
$\mu\in\Irr(H).$

The next result is similar to Lemma~2.8 in
\cite{Qian}.

\begin{lemma}\label{lem3} Let $N$ be a normal subgroup of a group $G$
and let $\theta\in\Irr(N)$ be a nonlinear character of $N.$ Suppose
that $\theta$ extends to $\chi\in\Irr(G).$ If $\chi$ is an
$\mi$ character of $G,$ then $\theta$ is also an $\mi$
character of $N.$
\end{lemma}

\begin{proof}

Assume that $\chi$ is an $\mi$ character of $G.$ Then there
exist a proper subgroup $U\lneq G,$ $\lambda\in\Irr(U)$ and $m\geq
1$ such that $m\chi=\lambda^G.$ Assume that
$T=\{r_1,r_2,\cdots,r_t\}$ is a set of representatives for the
double cosets of $U$ and $N$ in $G.$ As $\chi$ is an extension of
$\theta,$ we have that $\chi_N=\theta.$ By the discussion above, we
can and will assume that $U$ is maximal in $G.$
As $\lambda^G=m\chi$ and $\chi_N=\theta,$ we deduce that
$m\theta=(\lambda^G)_N.$ By Mackey's Lemma, we have that
$$(\lambda^G)_N=\sum_{j=1}^t ((\lambda^{r_j})_{U^{r_j}\cap N})^N$$
so $$m\theta=\sum_{j=1}^t ((\lambda^{r_j})_{U^{r_j}\cap
N})^N.$$ It follows that for each $j,$ we obtain that
$((\lambda^{r_j})_{U^{r_j}\cap N})^N$ is a multiple of $\theta.$ In
particular, we have that $(\lambda_{U\cap N})^N=k\theta$ for some
$k\geq 1.$

Assume first that $N\leq U.$ We then have that $\lambda_N=k\theta.$
By Lemma \ref{lem2}$(ii),$ we obtain that $\chi(1)=\theta(1)\geq
m\lambda(1)=mk\theta(1),$ which implies that $mk=1,$ hence $m=k=1.$
By Lemma \ref{lem2}$(ii)$ again, we have $|G:U|\lambda(1)=m\chi(1)$
and thus $|G:U|k\theta(1)=m\theta(1),$ which implies that $|G:U|=1,$
a contradiction.

Assume next that $N\nleq U.$ As $U$ is
maximal in $G,$ we obtain that $G=UN.$ Hence by Mackey's Lemma, we
have that $$(\lambda^G)_N=(\lambda_{U\cap N})^N=m\theta.$$ Since
$N\nleq U,$ we deduce that $U\cap N\lneq N.$ Let $\mu\in\Irr(U\cap
N)$ be an irreducible constituent of $\lambda_{U\cap N}.$ It follows
that $\mu^N=l\theta$ for some $l\geq 1.$ Hence $\theta$ is an $\mi$
character of $N$ induced from $(U\cap N,\mu)$ as required.
\end{proof}

\begin{lemma}\label{lem4}
Suppose that $N$ is a unique minimal normal nonabelian subgroup of a group $G.$ Assume that
$N=R_1\times R_2\times\cdots\times R_k,$ where each $R_i$ is isomorphic to a nonabelian simple
group $R,$ and $k\geq 1.$ Let $\theta$ be a nonlinear irreducible character of $R$ such that
$\theta$ extends to $\Aut(R).$ Let $\varphi=\theta^k\in\Irr(N).$ If $\chi\in\Irr(G)$ is any
extension of $\varphi$ to $G$ and $\chi$ is an $\mi$ character of $G,$ then $\theta$ is an $\mi$
character of $R.$
\end{lemma}

\begin{proof} Since $N\cong R^k$ is the unique minimal
normal subgroup of $G,$ we deduce that $G$ embeds into $\Aut(N)\cong
\Aut(R)\wr S_k,$ where $S_k$ denotes the symmetric group of
degree $k.$ As $\theta\in\Irr(R)$ extends to $\Aut(R),$ by \cite[Lemma~2.5]{Bianchi} we deduce that $\varphi=\theta^k\in \Irr(N)$ extends to
$G.$ Assume that $\chi\in \Irr(G)$ is an extension of $\varphi$ and that $\chi$ is an $\mi$ character of $G.$ By
Lemma \ref{lem3}, we deduce that $\varphi$ is an $\mi$ character of
$N$ induced from $(U,\lambda),$ where $U$ is a maximal subgroup of
$N,$ and $\lambda\in\Irr(U).$ Then $m\varphi=\lambda^N$ for some
$m\geq 1.$ As $N=R_1\times R_2\times\cdots\times R_k$ and $U$ is
maximal in $N,$ there exists $1\leq i\leq k$ such that $R_i\nleq U.$
Without loss of generality, we assume that $R_1\nleq U.$ Since $R_1\unlhd N,$ we
obtain that $N=R_1U,$ here we identify $R_1$ with $R_1\times
1\times\cdots \times 1\unlhd N.$ Observe that
$\varphi_{R_1}=\theta(1)^{k-1}\theta_1,$ where $\theta_1\in\Irr(R_1)$  is $N$-invariant. Since $N=R_1U,$ by Mackey's Lemma, we
obtain that $(\lambda^N)_{R_1}=\lambda_{U_1}^{R_1},$ where
$U_1:=R_1\cap U\lneq R_1.$ Then it follows from
$m\varphi_{R_1}=(\lambda^N)_{R_1}$ that
$m\theta(1)^{k-1}\theta_1=\lambda_{U_1}^{R_1}.$  Let
$\lambda_1\in\Irr(U_1)$ be any irreducible constituent of
$\lambda_{U_1},$ we then have that $\lambda_1^{R_1}=m_1\theta_1$ for
some $m_1\geq 1.$ Therefore we conclude that $\theta_1\in\Irr(R_1)$
is an $\mi$ character of $R_1.$ Hence $\theta$ is an $\mi$ character of
$R$ as wanted.
\end{proof}


\section{Finite simple groups of Lie type}\label{Lietypegroups}

In this section, we aim to prove Theorem \ref{th2} for simple groups of Lie type. Note that we will
consider the Tits group as a sporadic simple group rather than a simple group of Lie type and
exclude it from consideration in this section. Now it is well known that every simple group of Lie
type $S$ in characteristic $p$ possesses an irreducible character of degree $|S|_p,$ the size of
the Sylow $p$-subgroup of $S,$ which is called the Steinberg character of $S$ and is denoted by
$\St_S.$ (See \cite[Chapter~6]{Carter}.) Moreover the Steinberg character of $S$ is always
extendible to the full automorphism group $\Aut(S).$ (See for instance \cite[Theorem~2]{Bianchi}.)
Using the information on the character values of the Steinberg character given in \cite{Carter} and
also the classification  of the maximal subgroups of simple groups of Lie type satisfying certain
properties given in \cite{Liebeck,Malle}, we will prove that apart from some exceptions, the
Steinberg character cannot be an $\mi$ character. This is achieved in Lemma \ref{Steinberg}.
Finally, for these exceptions,  using \cite{atlas} we will find another nonlinear irreducible
character of $S$ which extends to $\Aut(S)$ but it is not an $\mi$ character.

We first draw some consequences under the assumption that the Steinberg character is an $\mi$
character. Recall that if $G$ is a group and $p\in\pi(G),$ then an element $g\in G$ is called
\emph{$p$-semisimple}  (or just \emph{semisimple} when $p$ is understood) whenever the order of $g$
is coprime to $p.$
\begin{lemma}\label{consequences} Let $S$ be a simple group of Lie type in characteristic $p.$ Suppose that $m\St_S=\lambda^S,$ where
$\lambda\in\Irr(H),m\geq 1$ and $H$ is a maximal subgroup of $S.$ Then the following hold.

\begin{enumerate}
\item $g^S\cap H\neq \emptyset\:\: \mbox{for any $p$-semisimple element $g\in S.$ }$
\item $p\nmid m \mbox{ and }\lambda(1)_p=|H|_p.$
\item $|S:H|_p\geq m\geq |S:H|_{p'}.$
\end{enumerate}
\end{lemma}

\begin{proof}
By \cite[Theorem~6.5.9]{Carter}, we have $\St_S(g)=\pm|C_S(g)|_p$ for any $p$-semisimple element
$g\in S.$ Thus for any $p$-semisimple element $g\in S,$ we obtain that $\lambda^S(g)=m\St_S(g)\neq
0.$  By the definition of induced characters, we obtain  $(1).$ For $g=1\in S,$ we have that
$m\St_S(1)=|S:H|\lambda(1).$ As $\St_S(1)=|S|_p$ and $\lambda(1)_p\mid |H|_p,$ we deduce that
$|S:H|_p\lambda(1)_p$ divides $|S|_p$ and so $|S:H|_p\lambda(1)_p=|S|_p$ as it is divisible by
$\St_S(1).$ This implies that $p\nmid m$ and $\lambda(1)_p=|H|_p,$ which proves $(2).$ Finally as
$m\St_S(1)=|S:H|\lambda(1),$ by applying $(2)$ we have $$m=|S:H|_{p'}\lambda(1)_{p'}\geq
|S:H|_{p'}.$$ By Lemma \ref{lem2}(ii), we obtain that  $m^2\leq |S:H|$ and so $$m^2\leq
|S:H|_p|S:H|_{p'}\leq m|S:H|_p,$$ which implies that $|S:H|_p\geq m\geq |S:H|_{p'}$ as required.
\end{proof}
The following result is a well known theorem due to Zsigmondy.
\begin{lemma}\emph{(See \cite[Theorems $5.2.14,5.2.15$]{KleidmanLiebeck}).}\label{Zsigmondy}
Let $q$ and $n$ be integers with $q\geq 2$ and $n\geq 3.$ Assume that $(q,n)\neq (2,6).$ Then $q^n-1$ has a prime divisor $\ell$ such
that

\begin{itemize}
\item $\ell$ does not divide $q^m-1$ for $m<n.$
\item If  $\ell\mid q^k-1$ then $n\mid k.$
\item $\ell\equiv 1~\mbox{\emph{(mod $n$)}}.$
\end{itemize}
\end{lemma}
Such an $\ell$ is called a \emph{primitive prime divisor}. We denote by
$\ell_n(q)$ the smallest primitive prime divisor of $q^n-1$ for fixed
$q$ and $n.$ When $n$ is odd and $(q,n)\neq (2,3)$ then there is a
primitive prime divisor of $q^{2n}-1$ which we denote by $\ell_{-n}(q).$

Let $S$ be a nonabelian simple group. If $S\unlhd G\leq \Aut(S),$ then $G$ is said to be an {\em
almost simple group} with socle $S$ and we write $\rm{soc}(G)$ to denote the socle $S$ of $G.$ We
refer to \cite[Chapter~4]{KleidmanLiebeck} for the detailed descriptions and definitions of the
geometric classes $\mathcal{C}_i(S)$ ($1\leq i\leq 8$) and the class $\mathcal{S}(S)$ of subgroups
of simple classical groups $S.$

Suppose that $S$ is a simple group of Lie type in characteristic $p$ and that $\St_S$ is an $\mi$
character of $S$ induced from the pair $(H,\lambda),$ where $H$ is a maximal subgroup of $S$ and
$\lambda\in\Irr(H).$ It follows from Lemma \ref{consequences}$(1)$ that $g^S\cap H\neq \emptyset$
for any semisimple element $g\in S$ and thus $\pi_{p'}(S)\subseteq \pi(H),$ where $\pi_{p'}(S)$ is
the set of all primes divisors of $|S|$ different from the characteristic $p.$ Hence two cases can
occur, either $\pi(H)=\pi(S)$ or $\pi(H)=\pi_{p'}(S).$  We first consider the case
$\pi(H)=\pi_{p'}(S).$ This implies that $|H|$ is prime to $p$ and that $H$ possesses some
semisimple element of certain maximal torus of $S.$ Hence we can apply \cite{Malle} for classical
groups and \cite{Liebeck} for exceptional groups of Lie type to obtain the possibilities for the
pairs $(S,H).$

\begin{lemma}\label{Liegroups1}
Let $S$ be a nonabelian simple group of Lie type in characteristic $p$ and let $H$ be a maximal
subgroup of $S.$ Suppose that $p\nmid |H|$ and that for any $p$-semisimple element $g\in S,$ we
have that $g^S\cap H\neq \emptyset.$ Then
$$(S,H)\in\{(\rm{L_2(5)},\rm{A}_4), (\rm{L}_2(5),\rm{S}_3), (\rm{L}_2(7), \rm{S_4}), (\rm{L}_3(2),7:3)\}.$$
\end{lemma}

\begin{proof}

It follows from the hypotheses that $\pi_{p'}(S)=\pi(H).$

{\bf Case $1.$} {$S$ is a simple classical group in characteristic $p$ defined over a field of size
$q=p^f$.} For each $p$-semisimple element $g\in S,$ by the hypotheses some conjugate of $g$ belongs
to $H.$ In particular $H$ possesses some element of the maximal torus of $S$ with order given in
\cite[Table~I]{Malle}. By \cite[Theorem~1.1]{Malle}, the following cases hold.

$\bf (A)$ $H\in \cup_{i=1}^8 \mathcal{C}_i(S).$ Then one of the following cases holds.

{\bf $(A_1)$} $H\in \mathcal{C}_1(S)$ and $S\in\{\textrm{L}^\epsilon_n(q), \textrm{O}_{2n+1}(q),
\textrm{O}_{2n}^+(q)\},$ where $n$ is at least $3,3,$ and $4,$ respectively. By inspecting the
orders of maximal subgroups in class $\mathcal{C}_1(S)$ in \cite[$\S4.1$]{KleidmanLiebeck} and
using the restriction on $n,$ we see that $p$ always divides $|H|.$ Hence this subcase cannot
happen.

{\bf $(A_2)$} $H\in \mathcal{C}_8(S)$ and $S\cong \textrm{S}_{2n}(q),$ where $n\geq 2.$ By
\cite[Proposition~4.8.16]{KleidmanLiebeck} we have that $H\cong \textrm{O}_{2n}^\epsilon(q)$ with
$q$ even. Obviously   $p$ always divides $|H|.$

{\bf $(A_3)$} $H\in \mathcal{C}_3(S)$ and $S\in\{\textrm{L}^\epsilon_n(q)(n\geq 3 \mbox{ odd}),
\textrm{S}_{2n}(q) (n\geq 2), \textrm{O}_{2n}^\epsilon(q)(n\geq 4)\}.$ By inspecting the orders of
maximal subgroups in class $\mathcal{C}_3(S)$ in \cite[$\S4.3$]{KleidmanLiebeck} and using the
restriction on $n,$ we see that $p$ always divides $|H|$ unless $S\cong \textrm{L}_n^\epsilon(q)$
where $n$ is an odd prime and $H$ is of type $\textrm{GL}_1^\epsilon(q^n)\cdot n.$

Assume that $n=3.$ By \cite[Proposition~4.3.6]{KleidmanLiebeck} , we have $|H|=3(q^2+\epsilon
q+1)/d,$ where $d=\gcd(3,q-\epsilon 1).$ In this case, $S$ has an element of order $(q^2-1)/d.$  It
follows that $(q^2-1)/d$ must divide $3(q^2+\epsilon q+1)/d,$ and so $(q^2-1)=(q-\epsilon
1)(q+\epsilon 1)$ divides $3(q^2+\epsilon q+1).$ As $\gcd(q+\epsilon 1,q^2+\epsilon q+1)=1,$ we
deduce that $q+\epsilon 1\mid 3,$ which implies that $q+\epsilon 1=1$ or $q+\epsilon 1=3$ since
$q+\epsilon 1>0. $ Solving these equations, we obtain that $q=2$ and $\epsilon=\pm$ or $q=4$ and
$\epsilon=-.$ Since $\textrm{U}_3(2)$ is not simple and $4^2-1 \nmid 3(4^2-4+1),$ we deduce that
$S\cong \textrm{L}_3(2)$ and  $H\cong 7:3.$

Now suppose that $n\geq 5$ is odd prime.  Assume first that both $\ell_{n-1}(q)$ and
$\ell_{\epsilon (n-2)}(q)$ exist. Observe that  these two primes are distinct. Then $S$ has
elements of orders $\ell_{n-1}(q)$ and $\ell_{\epsilon (n-2)}(q),$ respectively. Thus
$\ell_{n-1}(q)$ and $\ell_{\epsilon (n-2)}(q)$ divides $|H|.$ By Lemma \ref{Zsigmondy} neither
$\ell_{n-1}(q)$ nor $\ell_{\epsilon (n-2)}(q)$ can divide
$|\textrm{GL}_{1}^\epsilon(q^n)|=q^n-\epsilon 1$ since both $n-1$ and $n-2$ cannot divide $2n$ as
$n\geq 5$ is odd prime. As a result, both  $\ell_{n-1}(q)$ and $\ell_{\epsilon (n-2)}(q)$ must be
equal to the prime $n$ as $|H|\mid n(q^n-\epsilon 1),$ which is impossible. We now consider the
case when either $\ell_{n-1}(q)$ or $\ell_{\epsilon (n-2)}(q)$ does not exist. By Lemma
\ref{Zsigmondy}, we deduce that $S\cong \textrm{L}_7^\epsilon (2)$ or $S\cong \textrm{U}_5(2).$ If
the first case holds, then $\ell_{\epsilon 5}(2)\in \pi(S)$ exists. However we see that
$\ell_{\epsilon 5}(2)$ cannot divide $7(2^7-\epsilon 1)$ so $\ell_{\epsilon 5}(2)$ cannot divide
$|H|,$ a contradiction. For the latter case, we have $H\cong 11:5.$ But
$\pi_{2'}(\textrm{U}_5(2))=\{3,5,11\}\neq \pi(H).$

$\bf(B)$ $H\in \mathcal{S}(S).$ Then the following cases hold.

\begin{itemize}
\item[\bf($B_1$)] $(S,H)\in\{ (\textrm{L}_4(2),\textrm{A}_7),
(\textrm{U}_3(3),\textrm{L}_2(7)),(\textrm{U}_3(5),\textrm{A}_7),(\textrm{U}_4(3),\textrm{A}_7),$\\$
(\textrm{U}_4(3),\textrm{L}_3(4)),
(\textrm{U}_5(2),\textrm{L}_2(11)),(\textrm{U}_6(2),\textrm{M}_{22}),(\textrm{O}_7(3),\textrm{S}_9),(\textrm{S}_8(2),\textrm{L}_2(17))\}.$

\item[\bf ($B_2$)] $(S,\rm{soc}(H))=(\textrm{O}_8^+(q),\textrm{O}_7(q))$ with $q$ odd, or $(\textrm{O}_8^+(q),\textrm{S}_6(q))$ with $q$ even.
\end{itemize}

For these cases, we see that the characteristic $p$ of $S$ divides the order of $H.$

$\bf(C)$ $S\cong \textrm{L}_2(q),\textrm{U}_4(2)$ or $(S,H)\in\{ (
\textrm{L}_3(4),\textrm{L}_3(2)),( \textrm{S}_4(3),2^4\cdot  \textrm{A}_5),(
\textrm{O}_8^+(2),\textrm{A}_9)\}.$

{$(C_1)$} $S\cong \textrm{U}_{4}(2).$  As $\textrm{U}_4(2)\cong \textrm{S}_4(3),$ the
characteristic $p$ of $S$ is either $2$ or $3.$ Assume first that $p=2.$ In this case, $S$ contains
$2$-semisimple elements of order $5$ and $9.$ However by using \cite{atlas}, no maximal subgroup of
$S$ possesses two such elements simultaneously. Now assume that $p=3.$ In this case, by using
\cite{atlas} again we can check that $p$ divides the order of every maximal subgroup of $S.$

{$(C_2)$} $(S,H)\in\{ (  \textrm{L}_3(4),\textrm{L}_3(2)),( \textrm{S}_4(3),2^4\cdot
\textrm{A}_5),( \textrm{O}_8^+(2),\textrm{A}_9)\}.$  For these cases, the characteristic $p$ of $S$
divides the order of $|H|.$

{$(C_3)$} $S\cong \textrm{L}_{2}(q),$ where $q\geq 4.$

Assume first that $S\cong \textrm{L}_2(4)\cong \textrm{L}_2(5).$ By \cite{atlas}, every maximal
subgroup of $S$ is of even order so we can assume that $p=5.$ In this case, using \cite{atlas}
again, we deduce that $H\cong \textrm{S}_3$ or $\textrm{A}_4.$ Assume next that $S\cong
\textrm{L}_2(7)\cong \textrm{L}_3(2).$ By \cite{atlas}, we can see that if $p=2,$ then $H\cong 7:3$
and if $p=7,$ then $H\cong \textrm{S}_4.$ Assume that $S\cong \textrm{L}_2(9)\cong \textrm{A}_6$ or
$\textrm{L}_2(8).$ By \cite{atlas}, the order of every maximal subgroup of $S$ is divisible by $p.$
Assume that $S\cong \textrm{L}_2(q)$ where $q\in \{11,13\}.$ By \cite{atlas}, $S$ possesses
$p$-semisimple elements of order $(q+1)/2$ and $(q-1)/2$ respectively. However no maximal subgroups
of $S$ can possess both such elements simultaneously.

Thus we can assume that $q\geq 16.$ Since $H$ is a maximal subgroup of $S$ and $p\nmid |H|,$
inspecting the list of maximal subgroups of $\textrm{L}_2(q)$ in \cite{King}, the following cases
hold.

$(i)$ $H$ is a dihedral group of order $q+1,$ with $q$ odd.

$(ii)$ $H$ is a dihedral group of order $q-1,$ with $q$ odd.

$(iii)$ $H\cong \textrm{S}_4$ and $q\equiv \pm 1$ (mod $8$), $q$ prime or $q=p^2$ and $3<p\equiv \pm 3$ (mod $8$).

$(iv)$ $H\cong \textrm{A}_4$ and $q\equiv \pm 3$ (mod $8$) with $q>3$ prime.

$(v)$ $H\cong \textrm{A}_5$ and $q\equiv \pm 1$ (mod $10$), $q$ prime or $q=p^2$ and $p\equiv \pm 3$ (mod $10$).

It follows that $q\geq 17$ is odd and thus $S$ has two $p$-semisimple elements of order $(q\pm
1)/2$ so $H$ possesses elements of such orders. As $\gcd((q-1)/2,(q+1)/2)=1,$ we deduce that
$(q^2-1)/4$ divides $|H|.$ Since $q\geq 17,$ we can see that $(q^2-1)/4>60=|A_5|$ and that
$(q^2-1)/4>q+1$ and hence $H$ cannot be one of the groups given in $(i)-(v)$ above.

{\bf Case $2.$} {$S$ is a simple exceptional group of Lie type in characteristic $p$ with $S\neq
{}^2\textrm{F}_4(2)'$.} As $\pi_{p'}(S)=\pi(H),$ $|H|$ is divisible by all the primes in the second
column of \cite[Table~10.5]{Liebeck} so it follows from the proof of \cite[Theorem~4]{Liebeck} and
\cite[Table~10.5]{Liebeck}  that $S\cong \textrm{G}_2(q)$ with $q>2$ odd and $H\cong
\textrm{L}_2(13)$ where $\{\ell_3(q),\ell_6(q)\}=\{7,13\}$ and $p\neq 13.$ By
\cite[15.1]{Aschbacher}, $S$ possesses a cyclic maximal torus of  order $q^2-1.$ Now if $q=3$ then
$q^2-1=8.$ But then $\textrm{L}_2(13)$ has no element of order $8.$ Thus  $q\geq 4$ and hence
$q^2-1\geq 15$ which is strictly larger than any element orders in $\textrm{L}_2(13).$ Hence $H$
contains no element of order $q^2-1,$ a contradiction.
\end{proof}
We now consider the case $\pi(S)=\pi(H).$ In this case we can apply \cite[Corollary~5]{Liebeck} to obtain the possibilities for the pairs $(S,H).$

\begin{lemma}\label{Liegroups2} Let $S$ be a nonabelian simple group of Lie type in characteristic $p$ and let $H$ be a maximal
subgroup of $S.$ Suppose that $p\mid |H|$ and that for any $p$-semisimple element $g\in S,$ we have
that $g^S\cap H\neq \emptyset.$ Then one of the following cases holds.

$(1)$ $S\cong \rm{S}_4(3)$ and $H\cong 2^4:\rm{A}_5.$

$(2)$ $S\cong {\rm S}_{2n}(q)$ and $H\cong \Omega_{2n}^-(q)\cdot 2\cong \rm{SO}_{2n}^-(q),$ with $q,n$ even.

$(3)$ $S\cong {\Omega}_{2n+1}(q)$ and $H\cong \Omega_{2n}^-(q)\cdot 2,$ with $n\geq 2$ even and $q$ odd.

$(4)$ $S\cong {\rm O}_{2n}^+(q)$ and $H\cong \Omega_{2n-1}(q)$ with $n\geq 4$ even.

$(5)$ $S\cong {\rm S}_{4}(q)$ and $H\cong {\rm L}_2(q^2)\cdot 2$ with $q\geq 4$ even.
\end{lemma}

\begin{proof}
It follows from the hypotheses that $\pi(S)=\pi(H)$ and thus by \cite[Corollary~5]{Liebeck},  one
of the following cases holds.

$(i)$ $S\cong \textrm{U}_4(2)\cong \textrm{S}_4(3)$ and $H\cong 2^4:\textrm{A}_5.$ In this case,
the characteristic of $S$ is either $2$ or $3.$ If $p=3,$ then the pair
$(S,H)=(\textrm{S}_4(3),2^4:\textrm{A}_5)$ satisfies the hypotheses of the lemma. If $p=2,$ then
$S$ has a $2$-semisimple element of order $9$ but $H$ has no such element.

$(ii)$ $S\cong \textrm{L}_6(2)$ and $H\cong \textrm{L}_5(2),P_1$ or $P_5.$ As $H$ is maximal in
$S,$ we deduce that $H$ is isomorphic to the maximal parabolic subgroup $P_1$ or $P_5.$ Hence
$H\cong 2^5:\textrm{L}_5(2).$ In this case, $S$ possesses an element of order $63$ but $H$ contains
no such element.

$(iii)$ $S\cong \textrm{O}_8^+(2)$ and $H\cong \textrm{A}_9$ or $P_i,i=1,3,4.$ Note that $P_i\cong
2^6:\textrm{A}_8$ for $i=1,3,4.$ In any cases, $H$ has only one conjugacy class of elements of
order $5$ while $S$ has $3$ conjugacy classes of elements of order $5.$ Therefore these cases
cannot happen.

$(iv)$ $S\cong \textrm{S}_{2n}(q)$ with $n\geq 2, n,q$ even and $\Omega_{2n}^-(q)\unlhd H.$ In this
case, $H\in \mathcal{C}_8(S)$ and by \cite[Proposition~4.8.6]{KleidmanLiebeck}, we have that
$H\cong \textrm{SO}_{2n}^-(q)\cong \Omega_{2n}^-(q)\cdot 2.$

$(v)$ $S\cong \textrm{O}_{2n+1}(q)$ with $n\geq 3$ even, $q$ odd and $\Omega_{2n}^-(q)\unlhd H.$ In
this case, $H\in \mathcal{C}_1(S)$ and by \cite[Proposition~4.1.6]{KleidmanLiebeck}, we have that
$H\cong \Omega_{2n}^-(q)\cdot 2.$

$(vi)$ $S\cong \textrm{O}_{2n}^+(q)$ with $n\geq 4$ even and $\Omega_{2n-1}(q)\unlhd H.$ In this
case, $H\in \mathcal{C}_1(S)$ and by \cite[Proposition~4.1.6]{KleidmanLiebeck}, we have that
$H\cong \Omega_{2n-1}(q).$

$(vii)$ $S\cong \textrm{S}_{4}(q)$  and $\textrm{L}_2(q^2)\unlhd H.$ As $\textrm{S}_4(2)$ is not
simple, we assume that $q\geq 3.$   If $q=3,$ then $\rm{S}_4(3)\cong\Omega_5(3)\cong
\textrm{U}_4(2)$ and $H\cong \rm{S}_6.$ This case will be handled in $(viii)$ so we assume that
$q>3.$ Assume first that $q>3$ is odd. Using the isomorphism $\textrm{S}_4(q)\cong \Omega_5(q),$ it
follows from \cite[Theorem~1.1]{Malle} that $H\in \mathcal{C}_1(\Omega_5(q))$ and by
\cite[Proposition~4.1.6]{KleidmanLiebeck} we have that $H\cong \Omega_4^-(q)\cdot 2\cong
\rm{L}_2(q^2)\cdot 2.$ This possibility is included  in case $(3).$ Assume now that $q\geq 4$ is
even.  Then by \cite[Theorem~1.1]{Malle} again, $H\in \mathcal{C}_3(S) \cup \mathcal{C}_8(S)$ and
hence by \cite{KleidmanLiebeck} we obtain that $H\cong \Omega_4^-(q)\cdot 2\cong \rm{L}_2(q^2)\cdot
2.$ This is case $(5)$ in the lemma.

$(viii)$ The pair $(S,H)$ appears in Table \ref{Tab1}. For these cases, the conjugacy class of
$p$-semisimple elements with order given in the column `element order' in Table \ref{Tab1} does not
intersect $H.$\end{proof}
\begin{table}
 \begin{center}
  \caption{Simple groups of Lie type}\label{Tab1}
  \begin{tabular}{llc|llc}
   \hline
   $S$&$H$& element order&$S$&$H$& element order\\\hline
   $\textrm{L}_2(9)$&$\textrm{L}_2(5)$&$4$&$\textrm{S}_4(7)$&$\textrm{A}_7$&$8$\\
   $\textrm{U}_3(3)$&$\textrm{L}_2(7)$&$8$&$\textrm{G}_2(3)$&$\textrm{L}_2(13)$&$4$\\
   $\textrm{U}_3(5)$&$\textrm{A}_7$&$8$&$\textrm{U}_4(2)$&$\textrm{S}_6$&$9$\\
   $\textrm{U}_5(2)$&$\textrm{L}_2(11)$&$9$&$\textrm{S}_6(2)$&$\textrm{S}_8$&$9$\\
   $\textrm{U}_6(2)$&$\textrm{M}_{22}$&$9$&  $\textrm{U}_4(3)$&$\textrm{L}_3(4),\textrm{A}_7$&$8$\\\hline
  \end{tabular}

 \end{center}
\end{table}

As we will see shortly, only in case $(1)$ of the previous lemma is the Steinberg character an m.i
character. Cases $(2),$ $(4)$ and $(5)$ can be ruled out easily by using Lemma \ref{consequences}
and \cite{Curtis}. For case $(3)$  we will need more work.

We refer to \cite{KleidmanLiebeck,atlas} for the basic definitions and  properties of orthogonal
groups and their associated geometries. Let $p$ be an odd prime. Let $q$ be a power of $p$ and let
$\f_q$ be a finite field of size $q.$ Let $(V,\f_q,Q)$ be a classical orthogonal geometry with
$dimV=2n+1, n\geq 2$ and $Q$ a non-degenerate quadratic form on $V.$ For $x\in V-\{0\},$ a
one-space with representative $x$ is called a {\em point} in $V$ and is denoted by $\la x\ra.$ The
vector $x\in V-\{0\}$ is said to be {\em non-singular}  provided that $Q(x)\neq 0.$ Recall that a
non-singular point $\la x\ra$ with representative $x\in V-\{0\}$ is called a {\em plus point} (or
{\em minus point}) if $sgn(x^\perp)$ is $+$ (or $-$), respectively. (cf. \cite[page xi]{atlas}.) In
this situation, we also say that the non-singular point $\la x\ra$ is of plus or minus type if $\la
x\ra$ is a plus or minus point, respectively. Note that the group $\Omega(V)$ as defined in
\cite[(2.1.14)]{KleidmanLiebeck} is isomorphic to $\Omega_{2n+1}(q)\cong \textrm{O}_{2n+1}(q).$ In
this situation, $V$ is called the natural module for $\Omega(V).$ For $\xi\in\{\pm\},$ we define
$\mathfrak{E}_{\xi}(V)$ to be the set of all non-singular points of type $\xi$ in $V.$ For $\tau\in
\f_q,$ we define $$V_\tau=\{v\in V-\{0\}\:|\:Q(v)=\tau\}.$$


\begin{lemma}\label{transitive} Assume the set up above. Then the following hold.
\begin{itemize}
\item[(i)] Two
non-singular points $\la x\ra$ and $ \la y\ra$ have the same type if and
only if $Q(x)\equiv Q(y)\:(\mbox{mod $({\f}_q^*)^2$}).$
Indeed, for any non-singular point $\la z\ra$ with type $\zeta,$ we have $$\mathfrak{E}_{\zeta}(V)=
\{\la v\ra\subseteq V\;|\; Q(v)\equiv Q(z)\;  (\mbox{mod $({\f}_q^*)^2$}) \}.$$

\item[(ii)] For $\xi\in \{\pm\},$ $\Omega(V)$ acts transitively on $\mathfrak{E}_\xi(V).$

\item[(iii)] The stabilizers in $\Omega(V)$ of minus points form a unique conjugacy class of subgroups of $\Omega(V).$
\end{itemize}
\end{lemma}

\begin{proof} As $(iii)$ is a direct consequence of $(ii),$ we only need to prove $(i)$ and $(ii).$

For $(i),$ assume  that $Q(x)\equiv Q(y)\;(\mbox{mod $({\f}_q^*)^2$}).$ By \cite[Proposition
$2.5.4(ii)$]{KleidmanLiebeck}, we have that $\langle x\rangle$ and $\langle y\rangle$ are
isometric. By Witt's lemma \cite[Proposition~2.1.6]{KleidmanLiebeck}, this isometry extends to an
isometry $g$ of $V$ such that $\la x\ra g=\la y\ra.$ As $\la x\ra, \la y\ra$ are non-degenerate, we
obtain $x^\perp g=y^\perp.$ It follows that $x^\perp$ and $y^\perp$ are isometric, and hence
$sgn(x^\perp)=sgn(y^\perp),$ so $x$ and $y$ have the same type. Now assume that $x, y$ have the
same type. By Witt's lemma and \cite[Proposition $2.5.4(i)$]{KleidmanLiebeck}, there exists an
isometry between $x^{\perp}$ and $y^{\perp}.$ This isometry can extend to an isometry $g$ of $V$
such that $(x^{\perp})g=y^{\perp}.$ Since $(x^{\perp})^{\perp}=\langle x\rangle,$ and
$(y^{\perp})^{\perp}=\langle y\rangle,$ we deduce that $\langle x\rangle g=\langle y\rangle .$ Thus
$xg=\mu y$ for some $\mu\in {\f}_q^*.$ Therefore, $Q(x)=Q(xg)=Q(\mu y)=\mu^2Q(y).$ The other
statements are obvious. This proves $(i).$

For $(ii),$ since $Q(\mu x)=\mu^2Q(x)$ for $x\in V,$ $\mu\in \f_q^*$ and $Q(xg)=Q(x)$ for all $x\in
V,g\in\Omega(V),$ we see that $\Omega(V)$ acts on $\mathfrak{E}_\xi(V).$ Now fix a non-singular
point $\la x\ra$ of type $\xi.$ Let $\la v\ra$ be any non-singular point of the same type as that
of $\la x\ra.$ By $(i),$ we have that $Q(v)=\mu^2Q(x)$ for some $\mu\in\f_q^*.$ Let $y=\mu^{-1}v\in
\la v\ra.$ Then $Q(y)=Q(x)=:\tau$ and $\la x\ra =\la y\ra=\la v\ra.$ It follows that $x,y\in
V_{\tau}$ and hence by \cite[Lemma~2.10.5]{KleidmanLiebeck}, $\Omega(V)$ acts transitively on
$V_\tau$ so there exists $g\in\Omega(V)$ such that $xg=y.$ Therefore, we obtain that $\la x\ra
g=\la y\ra=\la v\ra,$ which means that $\Omega(V)$ is transitive on $\mathfrak{E}_\xi(V)$ as
wanted.
 \end{proof}

 \begin{remark}\label{rem}
In case $(3)$ of Lemma \ref{Liegroups2}, the maximal subgroup $H$ is exactly the stabilizer in
$S\cong\Omega_{2n+1}(q)$ of a minus point in the natural module for $S.$ By Lemma
\ref{transitive}(iii), there is only one class of such maximal subgroups in $S.$
 \end{remark}


We now consider the following set up. Let $n\geq 2$ be even and let $q$  be an odd prime power. Let
$S\cong \Omega_{2n+1}(q)$ with $S\neq \Omega_5(3)$ and let $H$ be the stabilizer of a minus point
in the natural module for $S.$ We deduce that $K:=H'\cong \Omega_{2n}^-(q)\leq S$ which is a normal
subgroup of index $2$ in $H.$ Let $\tilde{S}\cong \rm{Spin}_{2n+1}(q)$ and $\tilde{K}\cong
\rm{Spin}_{2n}^-(q).$ By the assumption on $n$ and $q,$ we deduce that the centers of both
$\tilde{S}$ and $\tilde{K}$ are cyclic of order $2.$ We know that $\tilde{K}$ possesses a cyclic
maximal torus $\tilde{T}$ of order $q^n+1$ which contains the center of $\tilde{K}$ and then by
factoring out the center $Z(\tilde{K}),$ we obtain a maximal torus $T$ of $K$ with order
$k:=(q^n+1)/2.$ (cf. \cite{Hiss}.) Let $\tilde{g}$ and $g$ be generators for $\tilde{K}$ and $K,$
respectively.

We know that the conjugacy classes of maximal tori of $\tilde{S}$ are parametrized by pairs of
partitions $(\alpha,\beta)$ of $n,$ that is, $\alpha=(\alpha_1,\alpha_2,\cdots,\alpha_r)$ and
$\beta=(\beta_1,\beta_2,\cdots,\beta_s)$ and $\sum \alpha_i+\sum \beta_j=n.$ The order of the
maximal torus parametrized by such pair $(\alpha,\beta)$  is given by
$$\prod_{\alpha_i}(q^{\alpha_i}-1)\prod_{\beta_j}(q^{\beta_j}+1).$$
The conjugacy classes of  the maximal tori of $\tilde{K}$ are also parametrized by pairs of partitions
$(\alpha,\beta)$ of $n,$ where $\beta$ has an odd number of parts, and the order of the maximal
torus parametrized by $(\alpha,\beta)$  is the same as in case $\tilde{S}.$ (cf. \cite{Malle99}).
It follows that $\tilde{T}$ is a maximal torus of $\tilde{K}$ parametrized by the pair of partition
$(\emptyset,(n)).$ Now by applying Lemma \ref{Zsigmondy} and the order formula of maximal tori of
$\tilde{K},$ we can easily see that the conjugacy class of $\tilde{K}$ containing $\tilde{T}$ is
the unique class whose order is divisible by $\ell_{2n}(q).$ Since $Z(\tilde{K})\leq \tilde{T}^x$
for all $x\in \tilde{K},$ we deduce that the conjugacy class of $T$ in $K$ is the unique conjugacy
class of maximal torus whose order is divisible by $\ell_{2n}(q).$ Since $\tilde{g}\in
\tilde{K}\leq \tilde{S}$ is semisimple, it lies in some maximal torus of $\tilde{S}.$  Using the
order formula of the maximal tori of $\tilde{S}$ given above, we deduce that $\tilde{g}$ must lie
in the Coxeter torus of $\tilde{S}$ (see \cite{Hiss}). Comparing the orders, we obtain that
$\tilde{T}$ is the Coxeter torus of $\tilde{S}$ so $T$ is also a maximal torus of $S.$ We also know
that $\tilde{T}$ has a regular element $\tilde{y}$, i.e., $C_{\tilde{S}}(\tilde{y})=\tilde{T}$
which implies that $\tilde{T}\leq C_{\tilde{S}}(\tilde{T})\leq C_{\tilde{S}}(\tilde{y})=\tilde{T}$
and hence  $C_{\tilde{S}}(\tilde{T})=\tilde{T}.$ As a consequence, we obtain that $C_S(T)=T$ which
in turn implies that $C_S(g)=\la g\ra$ and $C_H(g)=\la g\ra$ as $g\in K\leq H\leq S.$

\begin{lemma}\label{unique} Assume the set up above. Then the following hold.

\begin{itemize}

\item[(1)] If $T\leq U\leq S,$ where $U$ is maximal in $S,$  then $U$ is conjugate to $H$ in $S.$

\item[(2)] If $T^x\leq H$ for some $x\in S,$ then $T^x=T^u$ for some $u\in H.$

\item[(3)] We have $N_S(T)\leq H$ and $C_S(g)=\la g\ra=C_H(g).$

\item[(4)] We have that $g^S\cap H=g^H.$

\end{itemize}
\end{lemma}

\begin{proof}
Let $U$ be any maximal subgroup of $S$ containing $T.$ As $n\geq 2$ is even, we consider the case
$n=2$ and $n\geq 4$ separately. Assume first that $n\geq 4.$ By \cite[Theorem~1.1]{Malle}, we have
that $U\in \mathcal{C}_1(S).$ Since $n\geq 4$ and $q$ is odd, we deduce that $\ell_{2n}(q)$ exists
and divides $|T|$ so it divides $|U|.$ Using the descriptions of the subgroups in class
$\mathcal{C}_1(S)$ given in Propositions $4.1.6$ and $4.1.20$ in \cite{KleidmanLiebeck}, $U$ must
be the stabilizer of a minus point so $U$ is conjugate to $H$ in $S$ by Lemma
\ref{transitive}(iii). The remaining case can be argued similarly or we can check directly using
the list of maximal subgroups of $S$ given in \cite{King}. This proves $(1).$

Assume that $T^x\leq H$ for some $x\in S.$ It follows that $g^x\in H$ where $x\in S,$ $T=\la g\ra$
and the order of $g$ and $g^x$ is $k.$ Since $q$ is odd and $n$ is even, we have that $k$ is always
odd. As $K\unlhd H$ is of index $2,$ we deduce that $g^x\in K.$ As $k$ is prime to the
characteristic of $K\cong \Omega_{2n}^-(q),$ we obtain that $g^x\in K$ is semisimple and hence it
must lie in some maximal torus of $K$ whose order is divisible by $\ell_{2n}(q).$ Using the
discussion above, the conjugacy class of maximal torus of $K$ containing $T$ is the only class of
maximal torus of $K$ whose order is divisible by $\ell_{2n}(q)$ so   $T^x=\la g^x\ra$ is conjugate
to $T$ in $K$ and hence in $H.$ This proves $(2).$

We next show that $N_S(T)\leq H.$ Indeed if  $T\leq N_S(T)\nleq H,$ then $N_S(T)$ must lie in some
maximal subgroup of $S$ containing $T$ since $N_S(T)\neq S.$  By $(1)$ we have $T\unlhd N_S(T)\leq
H^x$ for some $x\in S.$ It follows that $T^{x^{-1}}\leq H$ and hence $T^{x^{-1}}=T^u$ for some
$u\in H$ by $(2).$ Thus $T^{ux}=T$ or equivalently $ux\in N_S(T)\leq H^x,$ which implies that
$ux=h^x,$ where $h\in H.$ Thus we conclude that $x=hu^{-1}\in H$ and so $N_S(T)\leq H^x=H.$ This
proves the first statement of $(3).$ The other statement has already been proved in the discussion
above.

Finally, since $g\in H,$ we obtain that $g^H\subseteq g^S\cap H.$ To prove the equality, it
suffices to show that if $g^x\in H,$ where $x\in S,$ then $x\in H.$ Suppose that $g^x\in H,$ where
$x\in S.$ Then $T^x=\la g\ra^x=\la g^x\ra\subseteq H$ and thus by $(2)$ we have that $T^x=T^u$ for
some $u\in H.$ It follows that $xu^{-1}\in N_S(T)$  and hence by $(3)$ we have $N_S(T)\leq H$ so
$xu^{-1}\in H,$ which implies that $x\in H$ as $u\in H.$ The proof is now complete.
\end{proof}


We now classify all simple groups of Lie type in which the Steinberg character is an m.i character.

\begin{lemma}\label{Steinberg}
Let $S$ be a simple group of Lie type in characteristic $p.$  If $H$ is a maximal
subgroup of $S$ such that $m\St_S=\lambda^S,$ where $\lambda\in\Irr(H)$ and $m\geq 1,$
then $$(S,H)\in\{( \rm{L}_2(5), \rm{A}_4),(\rm{L}_2(7), \rm{S}_4),(\rm{L}_3(2),7:3),(\rm{S}_4(3),2^4:\rm{A}_5)\}.$$
In the first three cases,
$\lambda\in\Irr(H)-\{1_H\}$ are chosen with $\lambda(1)=1.$ In the last case,
$\lambda\in\Irr(H)$ is chosen with $\lambda(1)=3.$ Furthermore, $m=1$ in all cases.
\end{lemma}

\begin{proof}
Assume that $m\St=\lambda^S,$ where $\lambda\in\Irr(H),m\geq 1$ and $H$ is a maximal subgroup of
$S.$ By Lemma \ref{consequences}(1) the hypotheses of Lemmas \ref{Liegroups1} and \ref{Liegroups2}
are satisfied so one of the following cases holds.

$(i)$ $(S,H)\in\{(\textrm{L}_2(5),
\textrm{A}_4),(\textrm{L}_2(5),\textrm{S}_3),(\textrm{L}_2(7),\textrm{S}_4),(\textrm{L}_3(2),7:3),(\textrm{S}_4(3),2^4:\textrm{A}_5)\}.$
Apart from the pair $(\textrm{L}_2(5),\textrm{S}_3),$ the remaining pairs satisfy the conclusion of
the lemma. This is done by using \cite{GAP}.

$(ii)$ $(S,H)=(\textrm{S}_{2n}(q),\Omega_{2n}^-(q)\cdot 2)$ with $n\geq 2$ and $n,q$ even. Let
$K=H'\cong \Omega^-_{2n}(q).$ Then $K$ is a normal subgroup of index $2$ in $H.$ Let
$\mu\in\Irr(K)$ be an irreducible constituent of $\lambda_K.$ Since $\lambda(1)_2=|H|_2=2|K|_2,$ we
have that $\lambda_K$ is not irreducible so $\mu(1)=\lambda(1)/2$ since $|H:K|=2,$ and thus
$\mu(1)=|K|_2.$ It follows that $\mu\in\Irr(K)$ is of $2$-defect zero and then by \cite{Curtis} the
only irreducible character of $2$-defect zero  of $\Omega_{2n}^-(q)$ with even $q$  is exactly the
Steinberg character, we deduce that $\mu=\St_K,$ where $K\cong \Omega_{2n}^-(q).$ By
\cite[Theorem~2]{Bianchi}, $\mu$ extends to $\mu_0\in\Irr(H)$ and hence by Gallagher's theorem
\cite[Corollary~6.17]{Isaacs}, $\psi\mu_0$ are all the irreducible constituents of $\mu^H,$ where
$\psi\in\Irr(H/K).$ Since $H/K$ is abelian of order $2,$ we obtain that all irreducible
constituents of $\mu^H$ are of the same degree $\mu(1).$ However this is a contradiction as
$\lambda(1)=2\mu(1)$ and $\lambda$ is also an irreducible constituent of $\mu^H$ by the Frobenius
reciprocity.

$(iii)$ $S=\textrm{O}^+_{2n}(q)$ with $n\geq 4$ even, and $H\cong \Omega_{2n-1}(q)$ when $q$ is odd
or $H\cong \textrm{S}_{2n-2}(q)$ when $q$ is even. We have
\[|S:H|=\frac{q^{n-1}(q^n-1)\gcd(2,q-1)}{\gcd(4,q^n-1)}.\] It follows that
\[|S:H|_p=q^{n-1} \mbox{ and } |S:H|_{p'}=\frac{(q^n-1)\gcd(2,q-1)}{\gcd(4,q^n-1)}.\] As $q\geq 3,$
we can check that $|S:H|_{p'}>|S:H|_p,$ contradicting Lemma \ref{consequences}$(3).$

 $(iv)$ $(S,H)=(\rm{O}_{2n+1}(q),\Omega_{2n}^-(q)\cdot 2)$ with $n\geq 2$ even, $q$ odd and $(n,q)\neq
(2,3).$ Let $K=H'\cong \Omega^-_{2n}(q).$ Then $K$ is a normal subgroup of index $2$ in $H.$ Since
$p$ is odd, we deduce that $\lambda(1)_p=|H|_p=|K|_p$ as $|H:K|=2.$ Let $\mu\in\Irr(K)$ be an
irreducible constituent of $\lambda_K.$ We know that $\lambda(1)/\mu(1)$ divides $|H:K|=2$ and thus
$\lambda(1)_p=\mu(1)_p=|K|_p,$ which means that $\mu$ is an irreducible character of $K$ with
$p$-defect zero. By applying the same argument as in Case $(ii),$ we deduce that $\mu=\St_K$ and
$\mu$ extends to $\mu_0\in\Irr(H).$ By Gallagher's theorem, $\psi\mu_0$ are all the irreducible
constituents of $\mu^H,$ where $\psi\in\Irr(H/K).$ As $H/K$ is abelian and $\lambda$ is an
irreducible constituent of $\mu^H,$ we obtain that $\lambda=\mu_0\tau$ for some $\tau\in
\Irr(H/K).$ It follows that $\lambda(1)=\mu_0(1)\tau(1)=\mu(1)$ so $\lambda$ is an extension of
$\mu$ to $H.$ In particular $\lambda(1)=|H|_p.$

By the discussion before Lemma \ref{unique}, $K$ possesses a cyclic maximal torus $T$ with
generator $g$ whose order $k$ is prime to $p$ so $g$ is semisimple. By Lemma \ref{unique}(3) we
have that $C_S(g)=C_H(g)=\la g\ra$ is a $p'$-group and hence by \cite[Theorem~6.5.9]{Carter} we
have that $$\St_S(g)=\pm|C_S(g)|_p=\pm 1$$ and $$\lambda(g)=\mu(g)=\St_K(g)=\pm|C_K(g)|_p=\pm 1.$$
By Lemma \ref{unique}(4), we have that $g^S\cap H=g^H$ and thus by \cite[p.~64]{Isaacs} we have
$$\lambda^S(g)=\frac{|C_S(g)|}{|C_H(g)|}\lambda(g)=\mu(g)=\pm 1.$$ As $m\St_S=\lambda^S,$ we obtain
that $m\St_S(g)=\lambda^S(g)$ and hence $m=1.$ By Lemma \ref{consequences}(3), we have that $m\geq
|S:H|_{p'}=(q^n-1)/2.$ As $q\geq 3$ and $n\geq 2,$  it is obvious that  $m>1,$ which contradicts
our previous claim that $m=1.$\end{proof}

Comparing the previous lemma with \cite[Lemma~2.4(2)]{Qian} and \cite[Lemma~6.3]{Konig}, we can see
that the Steinberg character of a simple group of Lie type is m.i if and only if it is imprimitive;
if and only if it is QSI, i.e., it is induced from a character $\varphi$ of a subgroup $U$ such
that $U/Ker(\varphi)$ is solvable.

We are now ready to prove the main result of this section.
\begin{proposition}\label{LieType} Theorem \ref{th2} holds for simple groups of Lie type.
\end{proposition}

\begin{proof}
Let $S$ be a simple group of Lie type  in characteristic $p.$ By way of contradiction, suppose that
every nonlinear irreducible character of $S$  which is extendible to $\Aut(S)$ is an $\mi$
character. As the Steinberg character of $S$ is extendible to $\Aut(S),$ it is an $\mi$ character.
By Lemma \ref{Steinberg}, one of the following cases holds.

$(1)$ $S\cong \textrm{L}_2(5)$ with $p=5.$ Since $\textrm{L}_2(5)\cong \textrm{L}_2(4),$ the
Steinberg character of $S$ with degree $|S|_2=4$ extends to $\Aut(S)$ but it is not an $\mi$
character by Lemma \ref{Steinberg}.

$(2)$ $S\cong \textrm{L}_2(7)\cong \textrm{L}_3(2)$ with $p=2$ or $p=7.$ In this case, both
irreducible characters of degrees $7$ and $8$ are $\mi$ characters. Using \cite{atlas}, the
irreducible character $\chi$ labeled by the symbol $\chi_4$ with degree $6$ of $S$ is extendible to
$\Aut(S).$ We will show that $\chi$ is not an $\mi$ character. Suppose that $\chi$ is an $\mi$
character of $S.$ Then $m\chi=\lambda^S,$ where $H$ is a maximal subgroup of $S,$
$\lambda\in\Irr(H)$ and $m\geq 1$ is an integer. Let $a$ and $b$ be elements in $S$ with order $2$
and $7$ respectively. By \cite{atlas}, we have that $\chi(a)\neq 0\neq \chi(b).$ Hence $a^S\cap
H\neq \emptyset\neq b^S\cap H.$ Thus $H$ possesses elements of orders $2$ and $7$ which implies
that $\{2,7\}\subseteq \pi(H).$ However by inspecting the list of maximal subgroups of $S$ in
\cite{atlas}, we see that no maximal subgroups of $H$ satisfies this property. This contradiction
shows that $\chi$ is not an $\mi$ character.

$(3)$ $S\cong \textrm{S}_4(3)$ with $p=3.$ In this case, we have that $S\cong \textrm{S}_4(3)\cong
\textrm{U}_4(2),$ the Steinberg character of $S$ with degree $|S|_2=2^6$ is extendible to $\Aut(S)$
but it is not an $\mi$ character by Lemma \ref{Steinberg}. The proof is now complete.
\end{proof}


\section{Alternating groups}\label{Alternatinggroups}

The main purpose of this section is to prove the following result.
\begin{proposition}\label{alternating} Theorem \ref{th2} holds for alternating groups of degree at least $7.$
\end{proposition}

\begin{proof}

Let $A_n$ act on the set
$\Omega=\{1,2,\cdots,n\}$ of size $n,$ where $n\geq 7.$ We follow
the argument in \cite[Lemma~3.1]{Konig}. Let $\chi_n\in\Irr(\textrm{A}_n)$
be an irreducible character of $\textrm{A}_n$ which is extendible to
$\Aut(S)\cong \textrm{S}_n$ with degree $n-1.$ In fact, we can choose
$\chi_n=\pi_n-1,$ where $\pi_n$ is the permutation character of the
natural action of $\textrm{A}_n$ on $\Omega.$ As $n\geq 7,$ the 2-point
stabilizer $\textrm{A}_{n-2}=\textrm{Stab}_{\textrm{A}_n}(\{1,2\})$ is doubly transitive
on $\Omega-\{1,2\}.$ As $\chi_n$ is an $\mi$ character, we have that
$m\chi_n=\lambda^{\textrm{A}_n}$ for some $\lambda\in\Irr(U)$ and $U\lneq
\textrm{A}_n.$ We have that $$(\chi_n)_{\textrm{A}_{n-2}}=\chi_{n-2}+2\cdot
1_{A_{n-2}}.$$
By Mackey's Lemma, we obtain that \[m\chi_{n-2}+2m
1_{\textrm{A}_{n-2}}=\sum_{x\in T}(\lambda^x_{U^x\cap
\textrm{A}_{n-2}})^{\textrm{A}_{n-2}},\] where $T$ is a representative set of
the double cosets of  $\textrm{A}_{n-2}$ and $U$ in $\textrm{A}_n.$  Hence $$(\lambda_{U\cap
\textrm{A}_{n-2}})^{\textrm{A}_{n-2}}=m_1\chi_{n-2}+m_21_{\textrm{A}_{n-2}},$$ where
$m_1,m_2\geq 1.$ By replacing $U$ with its conjugate, we can assume
that $m_2>m_1$ so $\textrm{A}_{n-2}\leq U$ as $(\lambda_{U\cap
\textrm{A}_{n-2}})^{\textrm{A}_{n-2}}$ takes only positive values. As $U$ is maximal
in $\textrm{A}_n$ and $\textrm{A}_{n-2}\leq U,$ we deduce that $U\cong \textrm{A}_{n-1}$
or $U\cong \textrm{S}_{n-2}.$ Assume that the latter case holds. Then the
conjugacy class of $\textrm{A}_n$ with representative $g\in \textrm{A}_n,$ where
$g=(1,2,\cdots,n)$ or $(1,2\cdots,n-3)(n-2,n-1,n)$ depending on
whether $n$ is odd or even, respectively, will intersect $U$ as
$\chi_n(g)\neq 0,$ which is impossible. Thus $U\cong \textrm{A}_{n-1}.$ As
$m\chi_n=\lambda^{\textrm{A}_n}$ and $|\textrm{A}_n:U|=n,$ by Lemma \ref{lem2}(ii),
we have that $m(n-1)=n\lambda(1)$  and $|\textrm{A}_n:U|=n\geq m^2.$
As $m(n-1)=n\lambda(1)$ and $\gcd(n,n-1)=1,$ we deduce that $n\mid m,$ hence $n\leq m,$ which is impossible as $n\geq m^2$ and $n\geq 7.$
\end{proof}

\section{Sporadic simple groups and the Tits group}\label{Sporadicgroups}

\begin{table}
 \begin{center}
  \caption{Sporadic simple groups and the Tits group} \label{spor}
  \begin{tabular}{llllcc}
   \hline
   $S$  & primes &$\chi$&$\chi(1)$& possible $H$&conjugacy class\\ \hline
   $\textrm{M}_{11}$&$11$&$\chi_9$&$45$&$\textrm{L}_2(11)$&$8A$\\
   $\textrm{M}_{12}$&$11$&$\chi_7$&$54$&$\textrm{L}_2(11),M_{11}$&$10A$\\
   $\textrm{J}_1$&$11,19$&$\chi_2$&$56$&$-$&\\
   $\textrm{M}_{22}$&$7,11$&$\chi_3$&$45$&$-$&\\
   $\textrm{J}_2$&$5,7$&$\chi_6$&$36$&$-$&\\
   $\textrm{M}_{23}$&$7,23$&$\chi_2$&$22$&$-$&\\
   $\textrm{HS}$&$7,11$&$\chi_{24}$&$3200$&$\textrm{M}_{22}$&$15A$\\
   $\textrm{J}_3$&$17,19$&$\chi_6$&$324$&$-$&$$\\
   $\textrm{M}_{24}$&$23$&$\chi_{22}$&$3520$&$\textrm{L}_2(23),\textrm{M}_{23}$&$21A$\\
   $\textrm{McL}$&$7,11$&$\chi_{12}$&$4500$&$\textrm{M}_{22}$&$12A$\\
   $\textrm{He}$&$7,17$&$\chi_{12}$&$1920$&$-$&$$\\
   $\textrm{Ru}$&$13,29$&$\chi_2$&$378$&$-$&$$\\
   $\textrm{Suz}$&$11,13$&$\chi_{43}$&$248832$&$-$&$$\\
   $\textrm{O'N}$&$19,31$&$\chi_{23}$&$175616$&$-$&$$\\
   $\textrm{Co}_3$&$7,23$&$\chi_5$&$275$&$\textrm{M}_{23}$&$24A$\\
   $\textrm{Co}_2$&$7,23$&$\chi_4$&$275$&$\textrm{M}_{23}$&$30A$\\
   $\textrm{Fi}_{22}$&$11,13$&$\chi_{54}$&$1360800$&$-$&$$\\
   $\textrm{HN}$&$11,19$&$\chi_{44}$&$2985984$&$-$&$$\\
   $\textrm{Ly}$&$37,67$&$\chi_2$&$2480$&$-$&$$\\
   $\textrm{Th}$&$19,31$&$\chi_4$&$27000$&$-$&$$\\
   $\textrm{Fi}_{23}$&$17,23$&$\chi_6$&$30888$&$-$&$$\\
   $\textrm{Co}_1$&$13,23$&$\chi_{40}$&$21049875$&$-$&$$\\
   $\textrm{J}_4$&$37,43$&$\chi_{15}$&$32307363$&$-$&$$\\
   $\textrm{Fi}_{24}'$&$23,29$&$\chi_6$&$1603525$&$-$&$$\\
   $\textrm{B}$&$31,47$&$\chi_{119}$&$2642676197359616$&$-$&$$\\
   $\textrm{M}$&$59,71$&$\chi_{16}$&$8980616927734375$&$-$&$$\\
   ${}^2\textrm{F}_4(2)'$&$5,13$&$\chi_{20}$&$1728$&$\textrm{L}_2(25)$&$10A$\\\hline
\end{tabular}
\end{center}
\end{table}
In this section, we will prove the following result.
\begin{proposition}\label{sporadic} Theorem \ref{th2} holds for sporadic simple groups and the Tits group.
\end{proposition}

\begin{proof}
Let $S$ be a simple sporadic group or the Tits group. All information that we need for the proof of
Proposition \ref{sporadic} is presented in Table \ref{spor}. For each sporadic simple group or the
Tits group $S,$ let $\pi_S$ be the set of primes in the second column of Table \ref{spor}. In the
fifth column, we list all the possibilities for the maximal subgroups $H$ of $S$ such that
$\pi_S\subseteq \pi(H).$ This is taken from \cite[Table~10.6]{Liebeck} for sporadic simple groups
and from \cite[Table~10.5]{Liebeck}  for the Tits group. Let $\chi\in\Irr(S)$ be the irreducible
character of $S$ labeled by the symbol in the third column of Table \ref{spor}. The corresponding
degree of $\chi$ is listed in the next column. The character $\chi$ is chosen to satisfy the
following properties.

\begin{enumerate}
  \item[(i)] $\chi$ is extendible to $\Aut(S);$
  \item[(ii)] For each prime $p\in\pi_S$ and any element $g_p\in S$ with order $p,$ we have that $\chi(g_p)\neq 0;$
  \item[(iii)] For any element $x\in S$ lying in the conjugacy class in the last column of Table \ref{spor},
  we also have that $\chi(x)\neq 0.$
\end{enumerate}

We now show that $\chi$ is not an $\mi$ character. By way of contradiction, assume that $\chi$ is
an $\mi$ character of $S.$ Then there exist a maximal subgroup $H\leq S$ and $\lambda\in\Irr(H)$
such that $m\chi=\lambda^S$ for some nonnegative integer $m.$ For each prime $p\in\pi_S$ and
$g_p\in S,$ by $(ii)$ we have  $\chi(g_p)\neq 0,$ so  $\lambda^S(g_p)=m\chi(g_p)\neq 0$, thus
$g_p^S\cap H\neq \emptyset.$ In particular $H$ possesses an element of order $p$ and thus
$\pi_S\subseteq \pi(H).$ The possibilities for $H$ are given Table \ref{spor}. In these cases, with
the same argument, we obtain that $x^S\cap H\neq \emptyset$ as $\chi(x)\neq 0$ by $(iii)$ and thus
$H$ has an element of order equal to that of $x.$ However we can see that $H$ has no elements with
such orders by checking \cite{atlas} directly.
\end{proof}

\section{Proofs of the main results}
\begin{proofoftheorem2}
This follows from Propositions \ref{LieType}, \ref{alternating}, \ref{sporadic} and the
classification of finite simple groups.\hfill$\square$
\end{proofoftheorem2}
It would be interesting if one could classify all $\mi$ characters of simple groups.
\begin{proofoftheorem1}
Assume that $N\unlhd G$ and that $G$ is an MI-group relative to $N.$ We show that $N$ is solvable.
Let $N\unlhd G$ be a  counterexample such that $|G|+|N|$ is minimal. It follows that $N$ is
nonsolvable.

We show that $N$ is the unique minimal normal subgroup of $G.$ We first show that $N$ is a minimal
normal subgroup of $G.$ Suppose not. Let $K\leq N$ be a minimal normal subgroup of $G.$ Then
$K\lneq N.$ By Lemma \ref{lem1}(ii), we obtain that $G$ is an MI-group relative to $K.$ Since
$|G|+|K|<|G|+|N|,$ by induction hypotheses we deduce that $K$ is solvable and hence $N/K$ is a
nonsolvable normal subgroup of $G/K.$ Now by Lemma \ref{lem1}(i), we have that $G/K$ is an MI-group
relative to $N/K.$ Thus by induction hypothesis again, we deduce that $N/K$ is solvable. Combining
with the previous claim, we obtain that $N$ is solvable, which is a contradiction. We have proved
that $N$ is a minimal normal subgroup of $G.$ Let $C=C_G(N).$ In order to show that $N$ is the
unique minimal normal subgroup of $G,$ it suffices to show that $C$ is trivial. Observe that
$C\unlhd G$ and thus by Lemma \ref{lem1}(i), we have that $G/C$ is an MI-group relative to $NC/C.$
As $N$ is nonsolvable and is a minimal normal subgroup of $G,$ we have that $N$ is isomorphic to a
direct product of some nonabelian isomorphic simple groups so $N\cap C=1.$ Therefore, $NC/C\cong
N/N\cap C\cong N$ is nonsolvable. If $C$ is nontrivial, then since $G/C$ is an MI-group relative to
$NC/C$, by induction hypothesis we deduce that $NC/C\cong N$ is solvable, which is impossible. Thus
$C$ is trivial. Hence $N$ is the unique minimal normal subgroup of $G$ as required.

Let $R$ be a nonabelian simple group such that $N= R_1\times R_2\times \cdots\times R_k,$ where
$R_i\cong R$ for all $i.$ By Lemma \ref{lem4}, every $\Aut(R)$-invariant nonlinear irreducible
character of $R$ is an $\mi$ character of $R.$ Now Theorem \ref{th2} provides a contradiction. The
proof is now complete. \hfill$\square$
\end{proofoftheorem1}

\begin{proofofcorollary2}
Let $L=\la H^G\ra$ be the normal closure of $H$ in $G.$  Then $H\leq L\unlhd G.$ By definition, $L$
is the smallest normal subgroup of $G$ containing $H.$ We show that $L$ is solvable. The remaining
statement is clear.

We first assume that $L=G.$ We show that $G$ is an $\MI$-group and then the result follows from
Corollary \ref{cor1}. Let $\chi\in\Irr(G)$ be any nonlinear irreducible character of $G.$ If $1_H$
is the only irreducible constituent of $\chi_H,$ then obviously $\chi_H=\chi(1)1_H$ and hence
$H\leq Ker(\chi)\unlhd G.$ Since $\chi$ is a nonlinear irreducible character of $G,$ we deduce that
$Ker(\chi)$ is a proper normal subgroup of $G$ containing $H,$ which is a contradiction as $\la
H^G\ra=G.$ It follows that $\chi_H$ has an irreducible constituent $\lambda\in \Irr(H)$ with
$\lambda\neq 1_H.$ By the hypotheses, we know that $\lambda^G$ is a multiple of some irreducible
character of $G.$ Since $(\lambda^G,\chi)_G=(\lambda,\chi_H)\neq 0$ by the Frobenius reciprocity,
we must have that $\lambda^G=m\chi$ for some nonnegative integer $m.$ Thus $\chi$ is an $\mi$
character. Hence $G$ is an $\MI$-group as wanted.

Now assume that $L\neq G.$ We claim that $G$ is an $\MI$-group relative to $L.$ Let
$\chi\in\Irr(G|L)$ and let $\theta$ be any irreducible constituent of $\chi$ upon restriction to
$L.$ Since $L\nleq Ker(\chi),$ we can choose $\theta\neq 1_L.$ If $\theta$ is not $G$-invariant,
then by the Clifford theory, we know that $\chi=\psi^G$ for some $\psi\in\Irr(I_G(\theta)|\theta)$
and hence $\chi$ is an $\mi$ character. Therefore, we can assume that $\theta$ is $G$-invariant. It
follows that $Ker(\theta)\leq L$ is a proper normal subgroup of $L$ since $\theta\neq 1_L.$ As in
the previous case, if $1_H$ is the only irreducible constituent of $\theta_H,$ then $H\leq
Ker(\theta)\lneq L\unlhd G,$ which is a contradiction as $L$ is the smallest normal subgroup of $G$
containing $H.$ Hence we conclude that $\theta_H$ possesses an irreducible  constituent
$\lambda\in\Irr(H)$ with $\lambda\neq 1_H.$ By the transitivity of induction, we obtain that $\chi$
is an irreducible constituent of $\lambda^G,$ and so by the hypotheses we deduce that
$\lambda^G=m\chi$ for some $m.$ Hence $\chi$ is an $\mi$ character. Thus $G$ is an $\MI$-group
relative to $L.$ Now the result follows from the main theorem.\hfill$\square$
\end{proofofcorollary2}

\subsection*{Acknowledgment} The authors are grateful to the referee for careful reading of the
manuscript and for his or her corrections and suggestions.

\end{document}